\documentclass{article}
\usepackage{amssymb,amsmath,amsthm,graphicx}
\usepackage[usenames,dvipsnames,svgnames,table]{xcolor}
\usepackage{color}
\newtheorem{theorem}{Theorem}

\newtheorem{proposition}[theorem]{Proposition}

\theoremstyle{definition}
\newtheorem{example}{Example}
\newtheorem{definition}{Definition}

\date{}

\title{\Large \textbf{Biquasiles and Dual Graph Diagrams}}

\author{Deanna Needell \footnote{Email: dneedell@cmc.edu. Partially supported by the Alfred P. Sloan Foundation and NSF CAREER $\#1348721$.} \and
Sam Nelson\footnote{Email: Sam.Nelson@cmc.edu. Partially supported by Simons Foundation collaboration grant $\#316709$.}}

\begin{document}
\maketitle

\begin{abstract}
We introduce \textit{dual graph diagrams} representing oriented knots and 
links. We use these combinatorial structures to define corresponding algebraic 
structures we call \textit{biquasiles} whose axioms are motivated by dual graph 
Reidemeister moves, generalizing the Dehn presentation of the knot group 
analogously to the way quandles and biquandles generalize the Wirtinger 
presentation. We use these structures to define invariants of oriented knots 
and links and provide examples.
\end{abstract}

\parbox{5.5in} {\textsc{Keywords:} biquasiles, dual graph diagrams, 
checkerboard graphs

\smallskip

}

\section{Introduction}

The checkerboard colorings of a planar knot complement have long been used
in knot theory, going back to papers such as \cite{L}. From the undecorated 
checkerboard graph, one can reconstruct an unoriented alternating knot or 
link up to mirror image. In \cite{K2}, signs are added to edges, enabling 
reconstruction of not necessarily alternating unoriented knots and links. In 
this paper, we introduce \textit{dual graph diagrams} for oriented knots and 
links, a type of diagram using both of the (mutually dual) checkerboard graphs 
decorated with some edges having signs and others having directions, enabling 
recovery of arbitrary oriented knots and links. A similar graph without 
decorations was used in the study of the dimer model of the Alexander and 
twisted Alexander polynomials in \cite{CDR}.

Analogously to the construction of quandles and biquandles from a coloring
scheme for arcs and semiarcs in oriented knot and link diagrams \cite{EN}, 
we introduce
a coloring scheme for vertices in a dual graph diagram. This coloring scheme
motivates a new algebraic structure known as a \textit{biquasile} with 
axioms determined by the dual graph Reidemeister moves. More precisely, the
biquasile axioms are chosen so that biquasile colorings of dual graph diagrams
are preserved faithfully by Reidemeister moves. This enables us to define 
biquasile counting invariants of knots and links and allows the introduction 
of enhancements of these invariants.

The paper is organized as follows. In Section \ref{DGD} we introduce dual graph
diagrams and the dual graph Reidemeister moves, as well as the reconstruction
algorithm; a generic dual graph diagram presents a type of directed bivalent 
spatial graph, sometimes known as a \textit{magnetic graph} \cite{KM,M,M2}. We 
introduce a geometric-style oriented link invariant defined from the dual graph
representation, the \textit{dual graph component number}, and show that this
invariant is bounded above by the braid index.
In Section \ref{B} we introduce biquasiles, deriving the biquasile axioms 
from the dual graph diagram Reidemeister moves and introduce biquasile
counting invariants for oriented knots and links. We provide several examples
illustrating the computation of the invariant and collecting results.
In Section \ref{AB} we turn our focus to the case of Alexander 
biquasiles, a type of biquasile structure defined on modules over
the three-variable Laurent polynomial ring  
$L=\mathbb{Z}[d^{\pm 1}, s^{\pm 1}, n^{\pm 1}]$. 
%In Section \ref{T} we generalize
%biquasiles to the case of coloring algebras for singular knots and links, 
%adding a new operation to obtain \textit{triquasiles}. 
We conclude in Section \ref{Q} with some questions for future research.

\section{Acknowledgements}

The authors would like to thank the referee for helpful comments and Jieon Kim
and Philipp Korablev for catching mistakes in the first version of this paper.

\section{Dual graph diagrams}\label{DGD}

We begin with a definition.

\begin{definition}
Let $D$ be an oriented knot or link diagram. The \textit{dual graph diagram}
$G$ associated to $D$ has a vertex associated to each region of the planar knot complement and edges
joining vertices whose regions are opposite at crossings. The edges are given
directions or $+/-$ signs as pictured:
\[\includegraphics{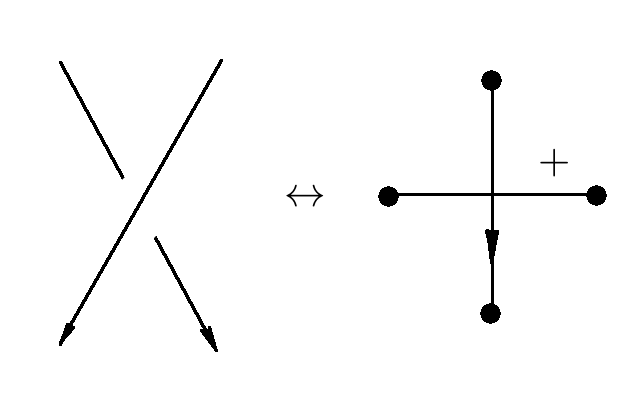}\quad \includegraphics{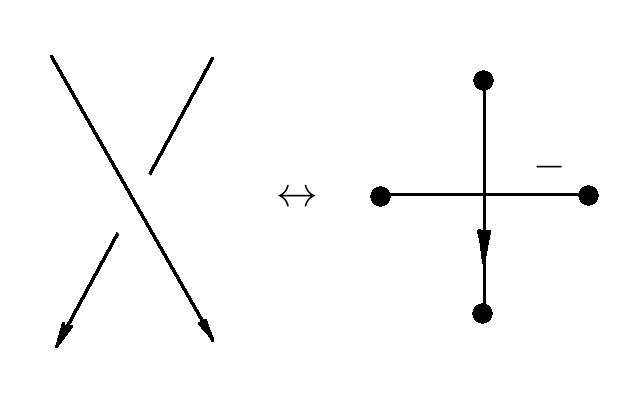}\]
Conversely, given a pair $G\cup G'$ of dual planar graphs (i.e., such that 
each region of $S^2\setminus G$ contains a unique vertex of $G'$ with adjacent
regions in $S^2\setminus G$ corresponding to adjacent vertices in $G'$), 
$G\cup G'$ becomes a dual graph diagram when we assign either a direction or 
$+$ or $-$ sign to each edge such that each pair of crossed edges has one 
signed edge and one directed edge.
\end{definition}

\begin{example}
The oriented knot diagram below has the corresponding dual graph diagram below.
\[\includegraphics{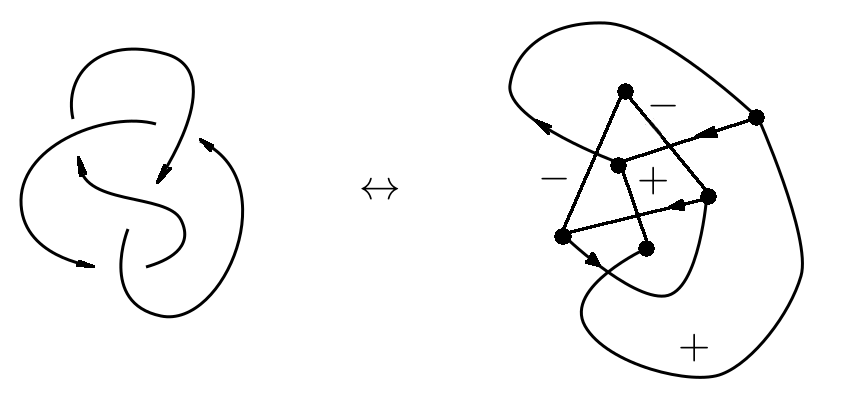}\]
\end{example}

We can understand the dual graph diagram as the result of superimposing the 
two checkerboard graphs associated to the knot or link diagram and decorating 
the edges to indicate orientation and crossing information. 
\[\includegraphics{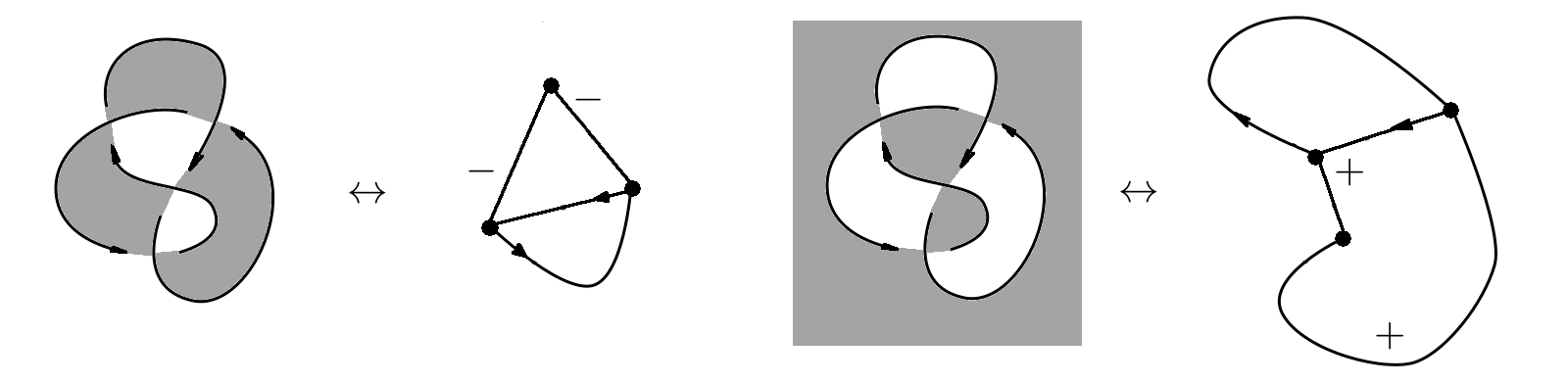}\]

A natural question is which dual graph diagrams present oriented knots
or links. Let us consider the reconstruction algorithm for obtaining the 
original oriented knot or link diagram from its dual graph diagram. First,
we note that each crossing of edges in the dual graph represents a 
crossing in the original link diagram, so we can start by putting a crossing 
at each edge-crossing with crossing information as determined by the direction
and sign decorations carried by the edges:
\[\includegraphics{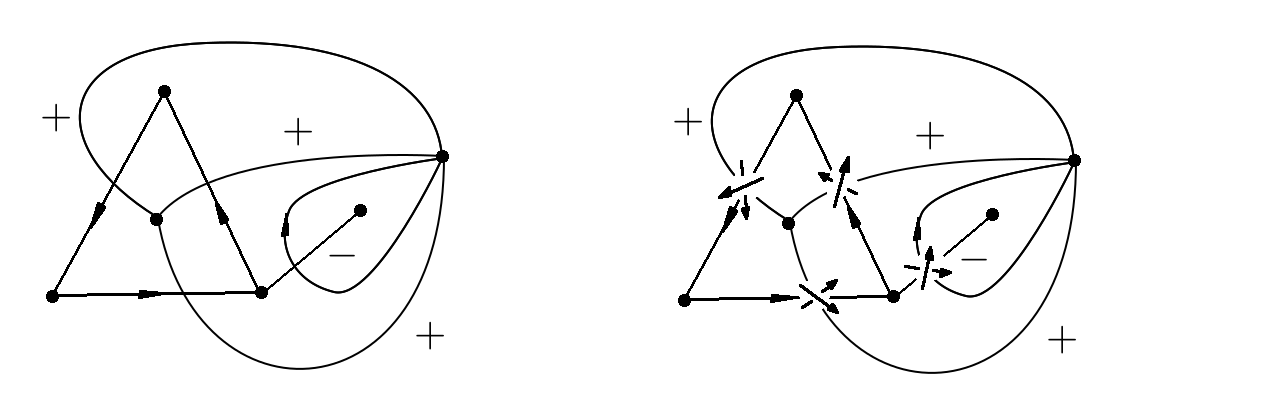}\]
We then observe that a dual graph diagram tiles the sphere $S^2$ with 
quadrilaterals with two corners given by vertices and two corners given by 
edge crossings as depicted:
\[\includegraphics{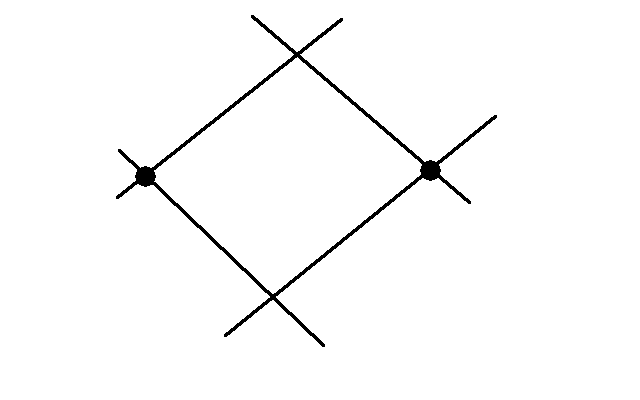}\]
Some such quadrilaterals may be degenerate, with boundary formed by a leaf and 
a loop:
\[\includegraphics{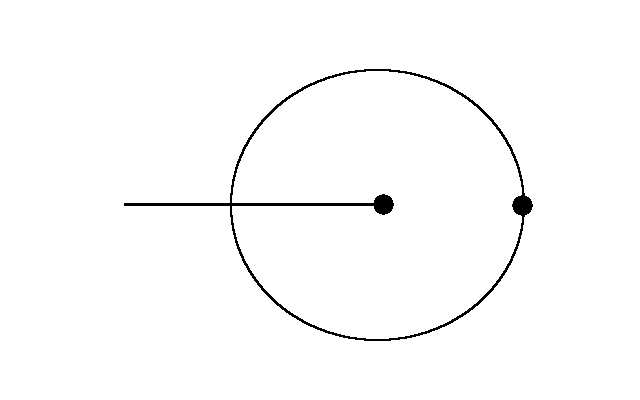}\]
Within each quadrilateral there is a unique path (up to planar isotopy)
connecting the ends of the crossings; drawing these paths completes the 
diagram:
\[\includegraphics{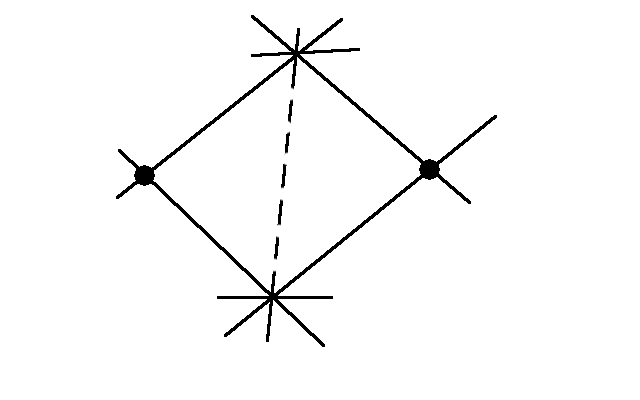}\]
We note that for some dual graph diagrams, the orientations at the ends of a 
strand may disagree; in these cases, we can include bivalent vertices in the 
interior of such an arc, obtaining a bivalent spatial graph with source-sink
orientations; such diagrams are known as \textit{magnetic graphs} in \cite{KM}:
\[\includegraphics{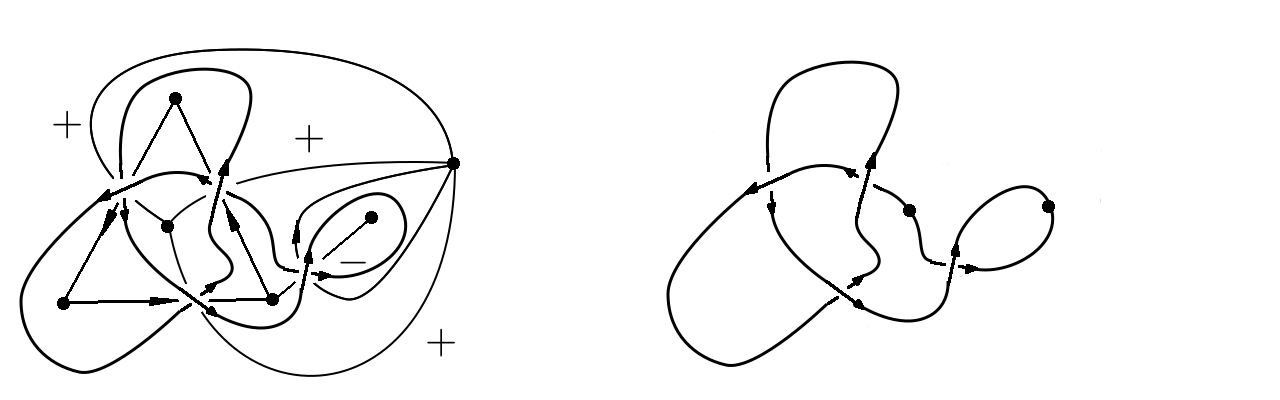}\]

We have the following:

\begin{proposition}\label{prop:1}
Every oriented knot or link has a dual graph diagram with directed edges 
consisting of disjoint cycles.
\end{proposition}

\begin{proof}
Simply put the knot or link diagram $L$ in closed braid form; the dual graph 
diagram then consists of $n$ disjoint directed cycles (where the braid 
being closed to form $L$ has $n+1$ strands) running vertically between 
the strands of the braid overlaid by locally horizontal signed edges.
\end{proof}

\begin{example}
The figure eight knot $4_1$ has closed braid presentation
with corresponding dual graph diagram as depicted:
\[\includegraphics{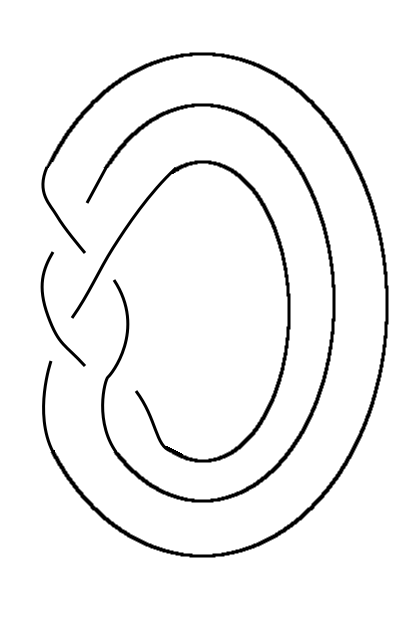}\qquad\qquad \includegraphics{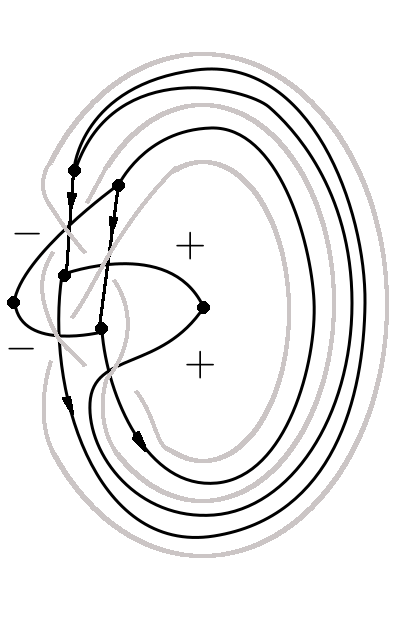}\]
\end{example}

\begin{definition}
The \textit{dual graph component number} of a knot or link $L$ is the minimal 
number of components in the directed subgraph of a dual graph diagram $D$ 
representing $L$, i.e. the graph obtained from $D$  by deleting the signed 
edges, taken over the set of all dual graph diagrams $D$ representing $L$.
\end{definition}

In light of the proof of Proposition \ref{prop:1}, we have the following 
easy observation:

\begin{proposition}
The dual graph component number of an oriented link is bounded above by the 
braid index.
\end{proposition}

Recall that two oriented knot or link diagrams represent ambient isotopic 
knots or links if and only if they are related by a sequence of 
\textit{oriented Reidemeister moves}:
\[\includegraphics{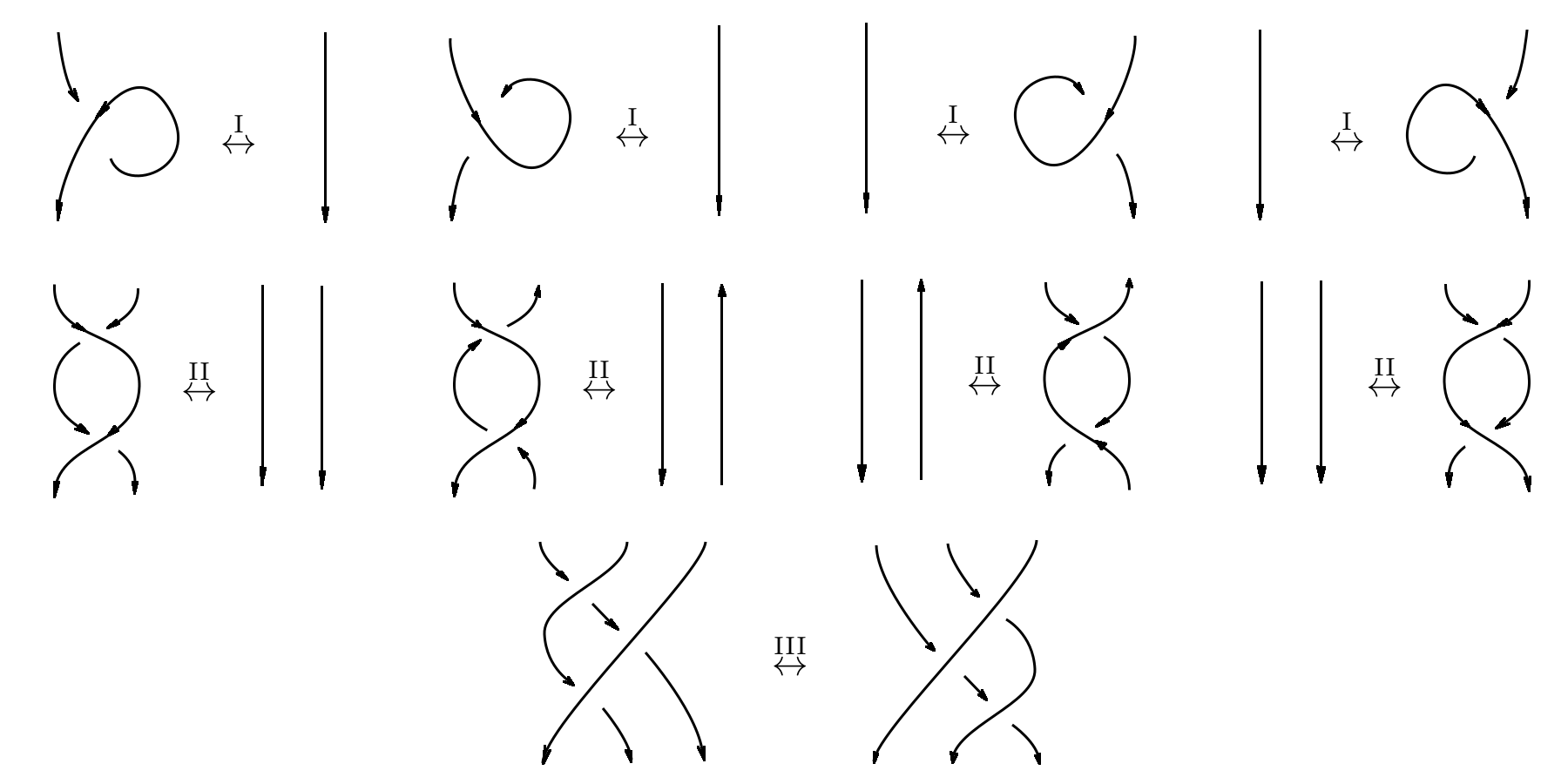}\]

Translating these Reidemeister moves into dual-graph format, we obtain the 
following moves:
\[\includegraphics{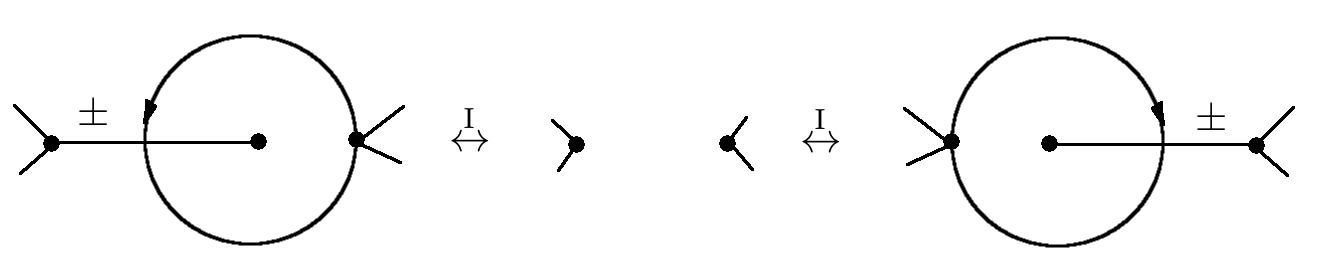}\]
\[\includegraphics{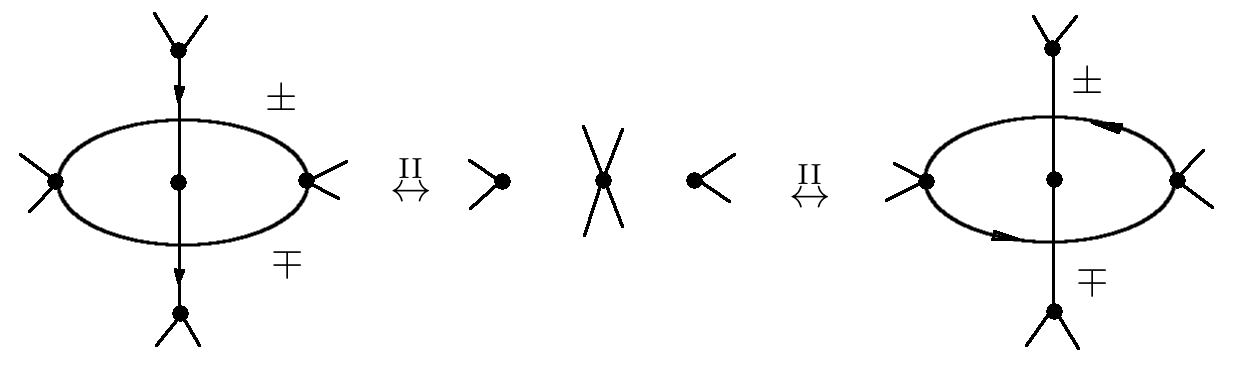}\]
\[\includegraphics{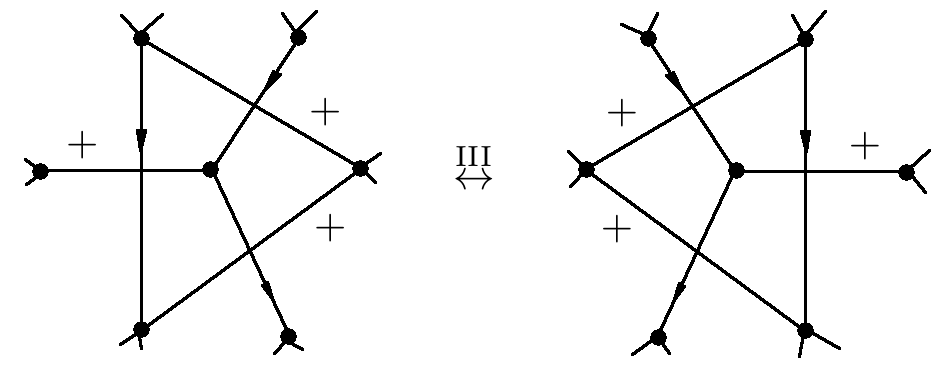}\]

Note that while dual graph Reidemeister moves do allow local vertex-introducing
and vertex-removing moves, they do not allow strands of the knot or link to 
move past bivalent vertices; thus, the larger category of dual graph diagrams
modulo dual graph Reidemeister moves is a slightly different category from 
the usual case of directed bivalent spatial graphs.
%Thus, we can identify Reidemeister equivalence classes of dual graph diagrams
%with magnetic graphs, and 
In particular, classical knots and links form Reidemeister equivalence classes 
of the subset of dual graph diagrams whose reconstructions do not require 
bivalent vertices, with moves restricted to forbid local orientation-reversing 
moves.  

\section{Biquasiles}\label{B}

Let $X$ be a set. We would like to define an algebraic structure on $X$ with
operations and axioms motivated by the Reidemeister moves on dual graph 
diagrams in order to define knot and link invariants. Let us impose on $X$
two binary operations, $\ast:X\times X\to X$ and $\cdot:X\times X\to X$ each 
defining a quasigroup structure on $X$, i.e., such that each operation has 
both a right and left inverse operation (not necessarily equal), which we will 
denote respectively by $/^{\ast},/,\backslash^{\ast}$ and $\backslash$. More 
precisely, we have the following definition:

\begin{definition}
Let $X$ be a set with binary operations $\ast, \cdot, \backslash^{\ast}, /^{\ast}, 
\backslash, /:X\times X\to X$ satisfying
\[\begin{array}{rcccl}
y\backslash^{\ast}(y\ast x) & = & x & = & (x\ast y)/^{\ast} y \\
y\backslash (y\cdot x) & = & x & = & (x\cdot y)/y. 
\end{array}\]
Then we say $X$ is a \textit{biquasile} if for all $a,b,x,y\in X$ we have
\[\begin{array}{rcll}
a\ast(x\cdot [y\ast(a\cdot b)]) & = & (a\ast[x\cdot y])\ast(x\cdot [y\ast([a\ast(x\cdot y)]\cdot b)]) & (i) \\
y\ast([a\ast (x\cdot y)]\cdot b) & = & (y\ast[a\cdot b])\ast([a\ast (x\cdot [y\ast(a\cdot b)])]\cdot b) & (ii).
\end{array}\]
\end{definition}

We will interpret these operations as the following vertex coloring rules at 
crossings in a dual graph diagram:
\[\includegraphics{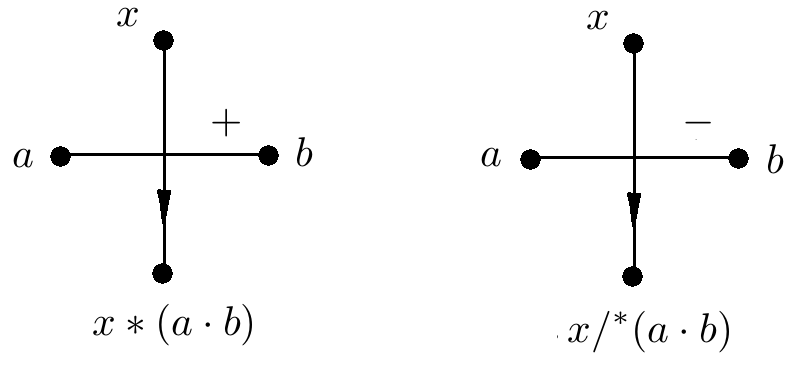}\]

\begin{definition}
Let $X$ be a biquasile and $D$ a dual graph diagram. Then any assignment of
elements of $X$ to the vertices of $D$ satisfying the conditions above is 
an \textit{$X$-coloring} of $D$.
\end{definition}

We want to establish axioms for our algebraic structure to ensure that for each
valid coloring on one side of the move, there is a unique valid coloring on 
the other side.

Consider the two Reidemeister I moves involving positively signed crossings:
\[\includegraphics{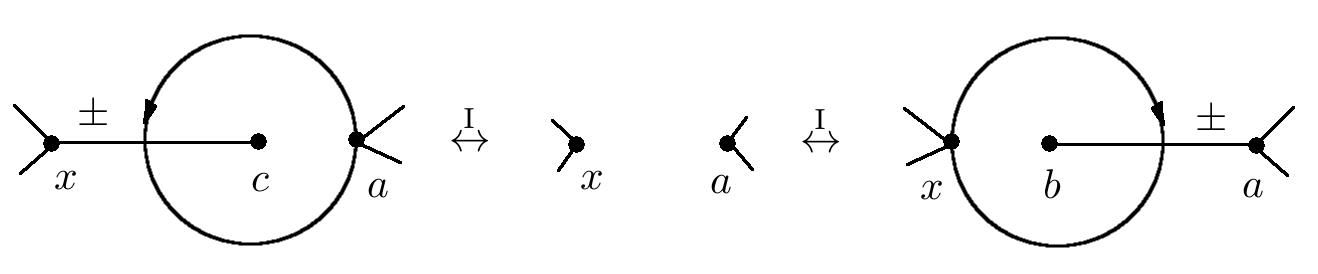}\]
The condition we need is that for all $x,a\in X$ there exist unique elements
$b,c\in X$ such that
\[\begin{array}{rcl}
a & = & a\ast(x\cdot c) \\
x & = & x\ast(b\cdot a). 
\end{array}\]
That is, we must be able to solve the equations $a  =  a\ast(x\cdot c)$
 and $x  =  x\ast(b\cdot a)$ for $c$ and $b$ respectively; this requires 
that $\ast$ has a left inverse operation $\backslash^{\ast}$ and that $\cdot$
has both a left and right inverse operation $\backslash$ and $/$; provided 
that $X$ is a quasigroup under both $\ast$ and $\cdot$, these conditions
are satisfied, with 
\[\begin{array}{rcl}
x\backslash(a\backslash^{\ast} a) & = & c \\
(x\backslash x)/a & = & b. 
\end{array}\]
Since the coloring rule at negative crossings is equivalent to
\[\includegraphics{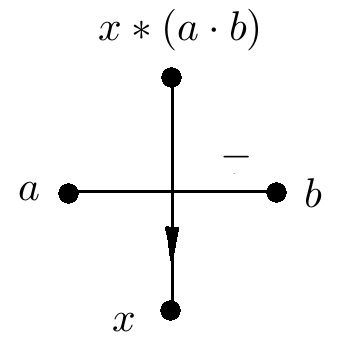}\]
the conditions arising from the Reidemeister I moves involving negatively
signed crossings are the same as those arising from the moves involving 
positively signed crossings.

Taking the \textit{direct} Reidemeister II moves (i.e., the type II moves with
both strands oriented in the same direction)  
\[\includegraphics{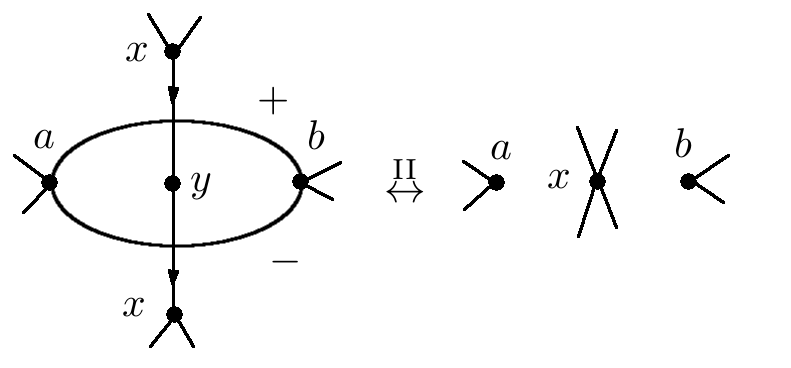}\]
we get the requirement that for every $x,a,b\in X$, there exists a unique 
$y\in X$ such that 
\[\begin{array}{rcll}
 y & = & x\ast(a\cdot b) & (dii.i)\\
 x & = & y/^{\ast}(a\cdot b) & (dii.ii);
\end{array}\]
then we have
\[[x\ast(a\cdot b)]/^{\ast}(a\cdot b)=x\]
as required by definition of the operations $\ast$ and $/^{\ast}$. 
The other direct II move is similar.

The \textit{reverse} Reidemeister II moves, i.e., the type II moves in which 
the strands are oriented in opposite directions,
\[\includegraphics{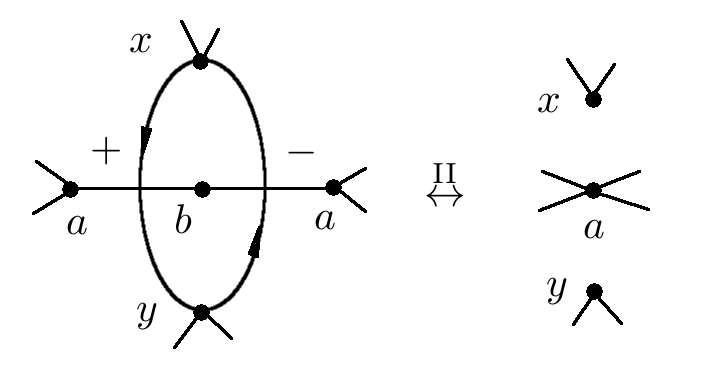}\]
require that for every $x,y,a\in X$ there exists a unique $b$ such that 
\[\begin{array}{rcll}
y & = & x\ast(a\cdot b) & (rii.i) \\
x & = & y/^{\ast}(a\cdot b) & (rii.ii). \end{array}\] 
The existence of the right and left inverse operations for $\ast$ and
$\cdot$ means we can solve both equations for $b$, obtaining 
$a\backslash(x\backslash^{\ast} y)$ in both cases:
\[\begin{array}{rcl}
y & = & x\ast(a\cdot b) \\
x\backslash^{\ast} y & = & a\cdot b \\
a\backslash(x\backslash^{\ast} y) & = & b
\end{array}\quad
\mathrm{and}
\quad\begin{array}{rcl}
x & = & y/^{\ast}(a\cdot b) \\
x\ast(a\cdot b) & = & y \\
a\cdot b & = & x\backslash^{\ast} y \\
b & = & a\backslash(x\backslash^{\ast} y)
\end{array}\]
as required.

The third Reidemeister move 
\[\includegraphics{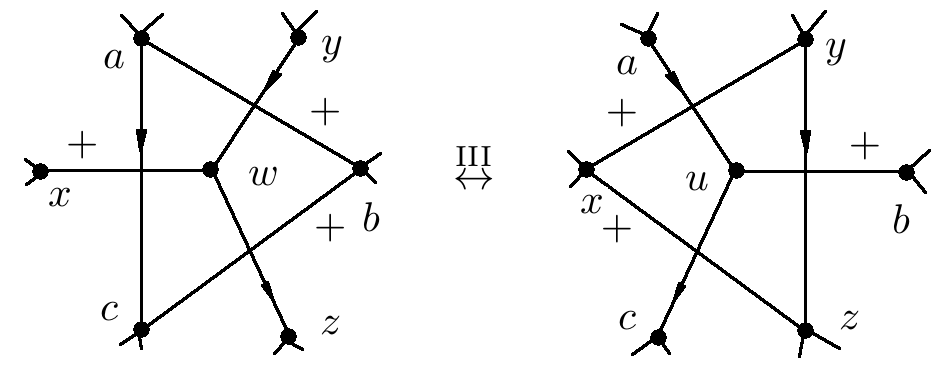}\]
yields the conditions
\[\begin{array}{rcl}
w & = & y\ast(a\cdot b) \\
c & = & a\ast(x\cdot w) \\
z & = & w\ast(c\cdot b)
\end{array}
\quad \mathrm{and}\quad
\begin{array}{rcl}
u & = & a\ast(x\cdot y) \\
c & = & u\ast(x\cdot z) \\
z & = & y\ast(u\cdot b)
\end{array}
\]
so we have
\[\begin{array}{rcl}
c & = & a\ast(x\cdot [y\ast(a\cdot b)]) \\
z & = & (y\ast(a\cdot b))\ast(c\cdot b)
\end{array}
\quad \mathrm{and}\quad
\begin{array}{rcl}
c & = & (a\ast(x\cdot y))\ast(x\cdot z) \\
z & = & y\ast([a\ast(x\cdot y)]\cdot b)
\end{array}
\]
whence
\[\begin{array}{rcll}
a\ast(x\cdot [y\ast(a\cdot b)]) & = & (a\ast[x\cdot y])\ast(x\cdot [y\ast([a\ast(x\cdot y)]\cdot b)]) & (i) \\
y\ast([a\ast (x\cdot y)]\cdot b) & = & (y\ast[a\cdot b])\ast([a\ast (x\cdot [y\ast(a\cdot b)])]\cdot b) & (ii)
\end{array}\]
must be satisfied for all $a,b,x,y\in X$. We can consider these to be somewhat 
complicated analogues of the distributive law. We can reformulate these 
slightly by defining functions
\begin{equation}\label{fg}
f_{a,b}(x,y)=
x\ast(a\cdot [b\ast(x\cdot y)])
\quad\mathrm{and}\quad
g_{a,b}(x,y)=y\ast([a\ast (x\cdot y)]\cdot b);
\end{equation}
then (i) and (ii) are the requirements that 
\begin{equation}\label{linreqs}
f_{a,b}(x,y)=f_{a,b}(x\ast(a\cdot b),y)\quad \mathrm{and}\quad
g_{a,b}(x,y)=g_{a,b}(x,y\ast(a\cdot b)).
\end{equation}

\begin{example} (Dehn Biquasile of an abelian group)
Let $A$ be any abelian group; then $A$ is a biquasile under the operations
\[a\cdot b= a+b\quad \mathrm{and}\quad x\ast y=y-x,\]
as we can easily verify:
\[a/b=a-b,\quad a\backslash b=b-a,\quad x/^{\ast}y=y-x,\quad \mathrm{and}\quad
x\backslash^{\ast}y=x+y \]
and
\[\begin{array}{rcl}
a\ast(x\cdot [y\ast(a\cdot b)]) & = & x+ b-y \\
& = & (a\ast[x\cdot y])\ast(x\cdot [y\ast([a\ast(x\cdot y)]\cdot b)])  \\
y\ast([a\ast (x\cdot y)]\cdot b) & = & 
x-a+b \\
& = & (y\ast[a\cdot b])\ast([a\ast (x\cdot [y\ast(a\cdot b)])]\cdot b).
\end{array}\]
Since the Dehn presentation relation $ax^{-1}by^{-1}=1 $ abelianizes to 
$y=a+b-x=x\ast(a\cdot b)$, this type of biquasile can be understood as a 
generalization of the Dehn presentation of the knot group.
\end{example}

As with other algebraic structures, we can specify a biquasile structure on 
a finite set $X=\{x_1,\dots,x_n\}$ with a pair of matrices encoding the
operation tables of the $\ast$ and $\cdot$ operations. More precisely,
the \textit{biquasile matrix} of a biquasile $(X,\ast,\cdot)$ of cardinality
$n$ is the $n\times 2n$ block matrix with $(j,k)$ entry $m\in\{1,2,\dots,n\}$ 
where
\[x_m=\left\{\begin{array}{ll}
x_j\ast x_k & 1\le k\le n \\
x_j\cdot x_k & n+1\le k\le 2n.
\end{array}\right.\]

\begin{example}\label{ex:X69}
Our \texttt{python} computations reveal $72$ biquasile structures on the set 
$X=\{x_1,x_2,x_3\}$ of three elements, including for instance
\[
\begin{array}{r|rrr}
\ast & x_1 & x_2 & x_3 \\ \hline
x_1 & x_3 & x_2 & x_1 \\
x_2 & x_2 & x_1 & x_3 \\
x_3 & x_1 & x_3 & x_2
\end{array}\quad
\begin{array}{r|rrr}
\cdot & x_1 & x_2 & x_3 \\ \hline
x_1 & x_3 & x_1 & x_2 \\
x_2 & x_1 & x_2 & x_3 \\
x_3 & x_2 & x_3 & x_1
\end{array}
\]
or more compactly
\[
\left[\begin{array}{rrr|rrr}
3 & 2 & 1 & 3 & 1 & 2 \\
2 & 1 & 3 & 1 & 2 & 3 \\
1 & 3 & 2  & 2 & 3 & 1
\end{array}\right].
\]
\end{example}

As with other algebraic categories, we have the following standard definitions.
\begin{definition}
A \textit{biquasile homomorphism} is a map $f:X\to Y$ between biquasiles
such that $f(x\ast x')=f(x)\ast f(x')$ and $f(x\cdot x')=f(x)\cdot f(x')$ 
for all $x,x'\in X$. A bijective homomorphism is an \textit{isomorphism}, 
and an isomorphism $f:X\to X$ is an \textit{automorphism}.
\end{definition}

\begin{example}
There are four biquasile structures on the set $X=\{x_1,x_2\}$, given by
\[X_1=\left[\begin{array}{rr|rr}
1 & 2 & 1 & 2 \\
2 & 1 & 2 & 1
\end{array}\right],\
X_2=\left[\begin{array}{rr|rr}
1 & 2 & 2 & 1 \\
2 & 1 & 1 & 2
\end{array}\right],\
X_3=\left[\begin{array}{rr|rr}
2 & 1 & 1 & 2 \\
1 & 2 & 2 & 1
\end{array}\right]\ \mathrm{and}\
X_4=\left[\begin{array}{rr|rr}
2 & 1 & 2 & 1 \\
1 & 2 & 1 & 2
\end{array}\right].
\]
Of these, there are two isomorphism classes, $\{X_1,X_4\}$ and $\{X_2,X_3\}$.
The 72 biquasiles of order three break down into 19 isomorphism classes, and 
our computations reveal 2880 biquasiles of order four comprising 177 
isomorphism classes.
\end{example}

\begin{definition}
A \textit{sub-biquasile} of $X$ is a subset $S\subset X$ which is itself
a biquasile under the operations of $X$; for $S\subset X$ to be a sub-biquasile
we need closure of $S$ under $\ast,\cdot$ and the right and left inverse
operations of both. Say a biquasile is \textit{simple} if it has no 
nonempty sub-biquasiles.
\end{definition}

\begin{example}
The biquasile structure on $X=\{x_1,x_2,x_3\}$ with operation matrix
\[
\left[\begin{array}{rrr|rrr}
3 & 2 & 1 & 3 & 1 & 2 \\
2 & 1 & 3 & 1 & 2 & 3 \\
1 & 3 & 2  & 2 & 3 & 1
\end{array}\right].
\]
in Example \ref{ex:X69} has one nontrivial sub-biquasile, 
the singleton set $\{3\}$. The biquasile with operation matrix
\[\left[\begin{array}{rrr|rrr}
1 & 3 & 2 & 2 & 1 & 3 \\
3 & 2 & 1 & 1 & 3 & 2 \\
2 & 1 & 3 & 3 & 2 & 1
\end{array}\right]\]
is simple since we have $1\cdot 1=2$, $2\cdot 2=3$ and $3\cdot 3=1$, so the
closure of every nonempty subset of $X=\{x_1,x_2,x_3\}$ under the biquasile 
operations is all of $X$.
\end{example}

\begin{definition}
Let $X$ be a set and $W(X)$ the set defined recursively by the following rules:
\begin{itemize}
\item[(i)] $x\in X$ implies $x\in W(X)$, and
\item[(ii)] $x,y\in W(X)$ implies the formal expressions
$x\ast y, x\cdot y, x/^{\ast}y, x\backslash^{\ast}y, x/y$ and
$x\backslash y \in W(X)$.
\end{itemize}
We call the elements of $W(X)$ \textit{biquasile words} in the 
\textit{generators} $X$. Then we define the \textit{free biquasile on $X$} to 
be the set of equivalence classes of biquasile words in $X$ modulo the relations
determined by the biquasile axioms, e.g. $(x\ast y)/^{\ast} y\sim x$, 
$y\ast([a\ast (x\cdot y)]\cdot b)\sim (y\ast[a\cdot b])\ast([a\ast (x\cdot [y\ast(a\cdot b)])]\cdot b)$, etc.
More generally, given a set of generators $X$ and a set of \textit{relations}
$R$, i.e., equations of biquasile words, the \textit{biquasile presented by}
$\langle X \ |\ R\rangle$ is the set of equivalence classes of biquasile words
in $X$ modulo the equivalence relation generated by the biquasile axioms
together with the relations in $R$.
\end{definition}

As in other universal algebraic systems, biquasiles presented by presentations 
related by the following \textit{Tietze moves} are isomorphic:
\begin{itemize}
\item[(i)] Adding or removing a generator $x$ and relation of the form $x=W$
where $W$ is a word in the other generators not involving $x$, and
\item[(ii)] Adding or removing a relation which is a consequence of the other
relations.
\end{itemize}

An important example is the \textit{fundamental biquasile} of an oriented
knot or link, defined as follows:

\begin{definition}
Let $L$ be a dual graph diagram and $X$ its set of vertices. Then the
\textit{fundamental biquasile} of $L$ is the biquasile with presentation
$\langle X\ |\ R\rangle$ where at each edge crossing in the dual graph 
diagram we have a relation as pictured.
\[\includegraphics{dn-sn-7.png}\]
\end{definition}

\begin{example}\label{ex:trefoil}
The dual graph diagram $D$ below has the fundamental biquasile presentation 
listed.
\[\raisebox{-0.6in}{\includegraphics{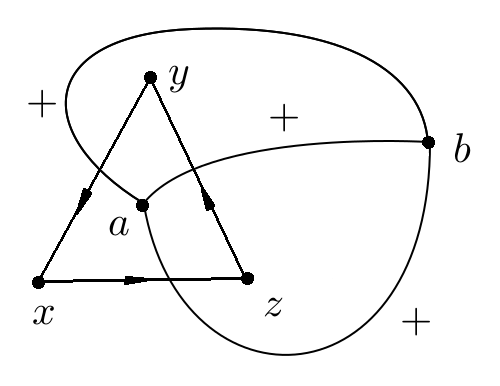}}
\quad \begin{array}{c}
\langle x,y,z,a,b \ |\ x\ast(b\cdot a)=z, z\ast(b\cdot a)=y, y\ast(b\cdot a)=x\rangle \\
=\langle x,y,a,b \ |\ [x\ast(b\cdot a)]\ast(b\cdot a)=y, y\ast(b\cdot a)=x\rangle \\
=\langle y,a,b \ |\ [(y\ast[b\cdot a])\ast(b\cdot a)]\ast(b\cdot a)=y \rangle.
\end{array}
\]
\end{example}

By construction, we have the following:
\begin{proposition}\label{prop:equiv}
The isomorphism class of the fundamental biquasile of an oriented link
is a link invariant.
\end{proposition}

\begin{definition}
In light of proposition \ref{prop:equiv}, we define the fundamental biquasile 
of an oriented knot or link $L$ to be the fundamental biquasile of any dual 
graph diagram $D$ representing $L$.
\end{definition}

\begin{definition}
Let $X$ be a biquasile and $D$ a dual graph diagram. The \textit{biquasile 
counting invariant} of $D$, denoted $\Phi_X^{\mathbb{Z}}(D)$, is the cardinality 
of the set of $X$-colorings of $D$, i.e., assignments of elements of $X$ to 
vertices in $D$ satisfying the conditions
\[\includegraphics{dn-sn-7.png}.\] Colorings of $D$ by $X$ can be interpreted
as homomorphisms from the fundamental biquasile of $D$ to $X$.
\end{definition}

By construction, we have the following:
\begin{theorem}
If $|X|$ is finite, then $\Phi_X^{\mathbb{Z}}(D)\le |X|^{|V|}$ where $V$
is the set of vertices in $D$. If $D$ and $D'$ are
related by Reidemeister moves, then $\Phi_X^{\mathbb{Z}}(D)=\Phi_X^{\mathbb{Z}}(D)'$
and hence  $\Phi_X^{\mathbb{Z}}(D)$ is an oriented link invariant.
\end{theorem}

\begin{example}
The dual graph diagram $D$ in Example \ref{ex:trefoil} has nine colorings
by the biquasile $X$ from example \ref{ex:X69}, as we can find by trying all
assignments of element of $X$ to the generators $y,a,b$ in the presentation
of the fundamental biquasile of $D$ and checking which such assignments satisfy
the relation $R$ given by $[(y\ast[b\cdot a])\ast(b\cdot a)]\ast(b\cdot a)=y$:
\[
\begin{array}{rrrr|rrrr|rrrr}
y & a & b & R? & y & a & b & R? & y & a & b & R?  \\ \hline
1 & 1 & 1 & \checkmark  & 2 & 1 & 1 &    & 3 & 1 & 1 &    \\
1 & 1 & 2 &    & 2 & 1 & 2 & \checkmark  & 3 & 1 & 2 &    \\
1 & 1 & 3 &    & 2 & 1 & 3 &    & 3 & 1 & 3 & \checkmark  \\
1 & 2 & 1 &    & 2 & 2 & 1 & \checkmark  & 3 & 2 & 1 &    \\
1 & 2 & 2 &    & 2 & 2 & 2 &    & 3 & 2 & 2 & \checkmark  \\
1 & 2 & 3 & \checkmark  & 2 & 2 & 3 &    & 3 & 2 & 3 &    \\
1 & 3 & 1 &  & 2 & 3 & 1 &    & 3 & 3 & 1 & \checkmark    \\
1 & 3 & 2 & \checkmark    & 2 & 3 & 2 &    & 3 & 3 & 2 &  \\
1 & 3 & 3 &    & 2 & 3 & 3 & \checkmark  & 3 & 3 & 3 &    \\
\end{array}
\]
\end{example}

\begin{example}
Using our Python code, we selected three biquasiles of order 5,
\begin{eqnarray*}
X_1 & = & \left[\begin{array}{rrrrr|rrrrr}%abqs(2,4,2,5)
1& 3& 5& 2& 4& 1& 5& 4& 3& 2\\ 
4& 1& 3& 5& 2& 3& 2& 1& 5& 4\\
2& 4& 1& 3& 5& 5& 4& 3& 2& 1\\
5& 2& 4& 1& 3& 2& 1& 5& 4& 3\\
3& 5& 2& 4& 1& 4& 3& 2& 1& 5
\end{array}\right] \\
X_2 & = & \left[\begin{array}{rrrrr|rrrrr}%abqs(1,1,3,5)
1& 4& 2& 5& 3& 1& 2& 3& 4& 5 \\
2& 5& 3& 1& 4& 2& 3& 4& 5& 1 \\
3& 1& 4& 2& 5& 3& 4& 5& 1& 2 \\
4& 2& 5& 3& 1& 4& 5& 1& 2& 3 \\ 
5& 3& 1& 4& 2& 5& 1& 2& 3& 4
\end{array}\right] \\
X_3 & = & \left[\begin{array}{rrrrr|rrrrr}%abqs(4,1,2,5)
1& 3& 5& 2& 4& 1& 2& 3& 4& 5 \\
5& 2& 4& 1& 3& 5& 1& 2& 3& 4 \\
4& 1& 3& 5& 2& 4& 5& 1& 2& 3 \\
3& 5& 2& 4& 1& 3& 4& 5& 1& 2 \\
2& 4& 1& 3& 5& 2& 3& 4& 5& 1
\end{array}\right]
\end{eqnarray*}
and computed the counting invariant 
for all prime knots with up to eight crossings and prime links with up to 
seven crossings; the results are collected below. Nontrivial 
values (i.e., values differing from that of the unlink of the same number 
of components) are listed in \textbf{bold}. 
\[\begin{array}{|r|rrrrrrrrrrrr|}\hline
K & 3_1 & 4_1 & 5_1 & 5_2 & 6_1 & 6_2 & 6_3 & 7_1 & 7_2 & 7_3 & 7_4 & 7_5 \\ \hline
\Phi_{X_1}^{\mathbb{Z}}(K) &  25 & 25 & 25 & 25 & 125 & 25 & 25& 25 & \mathbf{125} & 25 & 25 & 25  \\
\Phi_{X_2}^{\mathbb{Z}}(K) & 25 & \mathbf{125} & \mathbf{125} & 25 & 25 &25 & 25 & 25 & 25 & 25 & \mathbf{125} & 25 \\
\Phi_{X_3}^{\mathbb{Z}}(K) & 25 & 25 & 25 & 25 & 25 &25 & 25 & 25 & 25 & 25 & 25 & 25 \\\hline\hline
K                       & 7_6 & 7_7 & 8_1 & 8_2 & 8_3 & 8_4 & 8_5 & 8_6 & 8_7 & 8_8 & 8_9 & 8_{10} \\ \hline
\Phi_{X_1}^{\mathbb{Z}}(K) & \mathbf{125} & 25  & 25  & 25  & 25 & 25 & 25 & 25 & 25 & \mathbf{125} & 25 & 25 \\  
\Phi_{X_2}^{\mathbb{Z}}(K) & 25  & 25  & 25  & 25 & 25 & 25 & 25 & 25 & 25 & \mathbf{125} & \mathbf{125} & 25 \\
\Phi_{X_3}^{\mathbb{Z}}(K) &  25 & 25 & 25 & 25 & 25 & 25 & 25 &  25 & 25 & 25 & 25 & 25 \\ \hline\hline
\end{array}\]\[
\begin{array}{|r|rrrrrrrrrrrr|}\hline
K &  8_{11} & 8_{12} & 8_{13} & 8_{14} & 8_{15} & 8_{16} & 8_{17} & 8_{18} & 8_{19} & 8_{20} & 8_{21} &  \\ \hline
\Phi_{X_1}^{\mathbb{Z}}(K) & \mathbf{125} & 25 & 25 & 25 & \mathbf{125} & 25 & 25& 25 &  25 & 25 & 25 &\\
\Phi_{X_2}^{\mathbb{Z}}(K) &  25 & 25 & 25 & 25 & 25 & \mathbf{125} & 25 & \mathbf{125} & 25 & 25 & \mathbf{125} & \\
\Phi_{X_3}^{\mathbb{Z}}(K) &  25 & 25 & 25 & 25 & 25 & 25 & 25 & 25 & 25 & 25 & 25 & \\ \hline
\end{array}\]\[
\begin{array}{|r|rrrrrrrrr|}\hline
L & L2a1 & L4a1 & L5a1 & L6a1 & L6a2 & L6a3 & L6a4 & L6a5 & L6n1 \\ \hline
\Phi_{X_1}^{\mathbb{Z}}(K) & \mathbf{25} & \mathbf{25} & \mathbf{25} & \mathbf{25} & \mathbf{25} & \mathbf{25} & \mathbf{25} & \mathbf{25} & \mathbf{25} \\ 
\Phi_{X_2}^{\mathbb{Z}}(K) & \mathbf{25} & \mathbf{25} & \mathbf{25} & \mathbf{25} & 125 & \mathbf{25} & \mathbf{25} & \mathbf{25} & \mathbf{25} \\
\Phi_{X_3}^{\mathbb{Z}}(K) & 125 & 125 & 125 & 125 & 125 & 125 & \mathbf{25} & 625 & 625  \\ \hline\hline
L & L7a1 & L7a2 & L7a3 & L7a4 & L7a5 & L7a6 & L7a7 & L7n1 & L7n2 \\\hline
\Phi_{X_1}^{\mathbb{Z}}(L) & \mathbf{25} & \mathbf{25} & 125 & \mathbf{25} & 125 & \mathbf{25} & \mathbf{25} & 125& \mathbf{25} \\
\Phi_{X_2}^{\mathbb{Z}}(L) & \mathbf{25} & 125 & \mathbf{25} & \mathbf{25} & \mathbf{25} & \mathbf{25} & \mathbf{25} & 125& \mathbf{25}\\
\Phi_{X_3}^{\mathbb{Z}}(L) & 125 & 125 & 125 & 125 & 125 & 125 & 625 & 125 &125 \\ \hline
\end{array}
\]
\end{example}

\section{Alexander Biquasiles}\label{AB}

As with previous knot-coloring structures, we can consider the case of
biquasile structures with linear operations, which we call \textit{Alexander 
biquasiles.} We can think of Alexander biquasiles as generalizations
of Dehn biquasiles.

\begin{proposition}
Let $L=\mathbb{Z}[d^{\pm 1}, n^{\pm 1}, s^{\pm 1}]$. 
An $L$-module $X$ is a \textit{linear biquasile}, also called an 
\textit{Alexander biquasile}, under the operations
\begin{equation}\label{linops}
x\ast y=(-dsn^2)x+ny\quad \mathrm{and}\quad x\cdot y=dx+sy.
\end{equation}
\end{proposition}

\begin{proof}
First, we note that the invertibility of the variables $d,s$ and $n$
makes $\ast$ and $\cdot$ quasigroup operations.
Instate the notation of \eqref{fg}, where the operations are now given 
by \eqref{linops}.  We seek to verify the relations in \eqref{linreqs}.  
One readily computes that
\begin{align*}
f_{a,b}(x,y) = (dn)a + (-n^3s^2d)b + (s^2n^2)y  = %, \quad\mathrm{and}\\
f_{a,b}(x\ast(a\cdot b),y). %=  (dn)a + (-2n^3s^2d)b + (s^2n^2)y. 
\end{align*}
%Setting these equal, we thus require that:
%\begin{align}\label{req1}
%t + s^2uv &= t(t+s^2uv)\\
%0 &= su(t + s^2uv)\notag\\
%0 &= sv(t + s^2uv).\notag
%\end{align}
Similarly, we have
\begin{align*}
g_{a,b}(x,y) =  (-n^3d^2s)a + (ns)b + (d^2n^2)x = %, \quad\mathrm{and}\\
g_{a,b}(x,y\ast(a\cdot b)). %= (2tsu+s^3u^2v)a + (tsv+s^3uv^2+v)b + (s^2u^2)y. 
\end{align*}
%Setting these equal yields the same requirements as \eqref{req1}.  These are satisfied either when $t+s^2uv=0$ or when $t=1$ and $us=vs=0$.  
This verifies \eqref{linreqs} and completes the claim. %\dn{Confirm this does indeed complete the claim}
\end{proof}

\begin{example} \label{ex:zn}
\textit{(Alexander biquasile structures on $\mathbb{Z}_m$)}
%\dn{Not sure if we want this example...}
As a special case, one can consider groups $\mathbb{Z}_m$ and allow 
$d,n,s\in\mathbb{Z}_m^{\times}$.  For example, in $\mathbb{Z}_3$ there are 
seven unique (non-isomorphic) possibilities for assigning $d,n,s$ that 
satisfy \eqref{linreqs}:
\begin{center}
\begin{tabular}{llll}
 \bfseries{d} & \bfseries{n} & \bfseries{s} & \bfseries{$-n^2ds$} \\
\hline
 1 &	1 &	2 & 1\\
2&	1	&1 & 1\\
2&	2	&1 & 1\\
1	&1&	1 & 2\\
1	&2	&1& 2\\
2	&1	&2& 2\\
2	&2	&2& 2\\
\end{tabular}
\end{center}
We may also be interested in the number of configurations of $(d,n,s)$ that 
satisfy the conditions \eqref{linreqs} (as well as forming a quasigroup), and 
how many of those configurations create non-isomorphic structures:
\begin{center}
\begin{tabular}{|l|l|l|}
\hline
\bfseries{m} & \bfseries{\# configurations} & \bfseries{\# non-isomorphic} \\
\hline
2	& 1 &	1 \\
3 & 8 & 7\\
4& 8&7 \\
5& 64&34 \\
6& 8&7 \\
7& 216&137 \\
8&64 &33 \\
9& 216& 152\\
10& 64& 34\\
\hline
\end{tabular}
\end{center}

As an extension to this example, we consider a dual graph diagram $K$ 
%over $\mathbb{Z}_3$ 
and compute the number of colorings, $\Phi_{\mathbb{Z}_3}^{\mathbb{Z}}(K)$, of $K$
by the biquasile $X=\mathbb{Z}_3$ with the operations
\[
x\ast y=x+y, \quad x/^{\ast} y=x+2y, \quad x\cdot y=x+2y,
\]
which correspond to the selection of $d=1, n=1,$ and $s=2$.  Consider the following knot and corresponding dual graph diagram $K$, where we have labeled the nodes for reference.

\begin{center}
\includegraphics{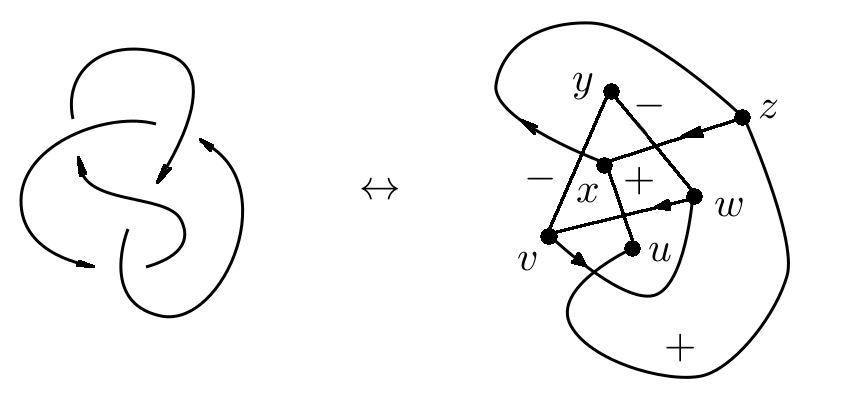}
\end{center}

This dual graph diagram yields the following four equations (along with their equivalent forms):\\

\begin{center}
\begin{tabular}{rcl}
$z/^{\ast}(y\cdot w) = x$ &$\quad\leftrightarrow\quad$ & $z + 2y + w = x$\\
$x/^{\ast}(y\cdot v) = z$ &$\quad\leftrightarrow\quad$ & $x + 2y + v = z$\\
$w\ast(x\cdot u) = v$ &$\quad\leftrightarrow\quad$ & $w + x + 2u = v$\\
$v\ast(z\cdot u) = w$ &$\quad\leftrightarrow\quad$ & $v + z + 2u = w$.\\
\end{tabular}\\
\end{center}

Rewriting as homogeneous equations and putting in matrix form (with respect to the vector $[u,v,w,x,y,z]$), these give:
$$
\left[
\begin{array}{rrrrrr}
0 & 0 & 1 & 2 & 2 & 1 \\
0 & 1 & 0 & 1 & 2 & 2 \\
2 & 2 & 1 & 1 & 0 & 0 \\
2 & 1 & 2 & 0 & 0 & 1
\end{array}\right]
\xrightarrow{\mathrm{Row\ moves\ over\ \mathbb{Z}_3}}
\left[\begin{array}{rrrrrr}
1 & 1 & 0 & 1 & 2 & 1 \\
0 & 1 & 0 & 1 & 2 & 2 \\
0 & 0 & 1 & 2 & 2 & 1 \\
0 & 0 & 0 & 1 & 0 & 2 \\
\end{array}\right]
$$

The solution space is thus two-dimensional, and so we have 
$\Phi^\mathbb{Z}_{\mathbb{Z}_3}(K) = 3^2 = 9$.

%\dn{Please check my algebra here. Also, do we want to include another example?}
\end{example}

\begin{example}
As in the case of quandles and biquandles \cite{EN}, we can define a 
module-valued Alexander invariant of oriented knots and links by considering
the \textit{Alexandrization} of the fundamental biquasile of an oriented
knot or link, i.e. the fundamental biquasile written as an Alexander biquasile. More precisely, the kernel of the coefficient matrix of the 
homogeneous system of linear equations over 
$L$ is an $L$-module valued invariant of 
oriented knots and links analogous to the classical Alexander invariant;
from it, we can derive polynomial-valued invariants via the Gr\"obner basis
construction described in \cite{CHN}. 

For instance, the knot $4_1$
in example \ref{ex:zn} has Alexander biquasile given by the kernel
of the matrix below with entries in $L$:
\[
%\left[
%\begin{array}{rrrrrr}
%0 & 0 & s &-n & d & -dsn^2 \\
%0 & s & 0 & -dsn^2 & d &-n \\
%s & -n & -dsn^2 & d & 0 & 0 \\
%s & -dsn^2 & -n & 0 & 0 & d
%\end{array}\right].
\left[
\begin{array}{rrrrrr}
-dsn^2 & nd & -1 & ns & 0 & 0 \\
-1 & nd & -dsn^2 & 0 & ns & 0 \\
nd & 0 & 0 & ns & -1 & -dsn^2 \\
0 & 0 & nd & ns & -dsn^2 & -1
\end{array}\right].
\]
These invariants will be the subject of 
another paper.
\end{example}

\section{Questions}\label{Q}

We end with a few collected questions for future research.

\begin{itemize}
\item What, if anything, is the relationship between biquasiles and biquandles? 
\item What enhancements  of the biquasile counting invariants can be defined? 
\item What kinds of categorifications of biquasiles and their invariants are 
possible?
\end{itemize}

\bibliography{dn-sn}{}
\bibliographystyle{abbrv}

\bigskip

\noindent
\textsc{Department of Mathematical Sciences \\
Claremont McKenna College \\
850 Columbia Ave. \\
Claremont, CA 91711}

\end{document}